\documentclass[12pt,article]{amsart}
\usepackage{amsmath,color}
\usepackage{amssymb}
\newtheorem{corollary}[equation]{Corollary}

\newtheorem{claim}[equation]{\indent{\it Claim}\rm }

\newtheorem{lemma}[equation]{Lemma}

\newtheorem{theorem}[equation]{Theorem}

\newcommand{\C}{{\mathbb{C}}}

\renewcommand{\P}{{\mathbb{P}}}

\newcommand{\R}{{\mathbb{R}}}

\newcommand{\supp}{\mathrm{Supp}\,}

\numberwithin{equation}{section}

\title[Total curvature of a complete minimal surface]{Total curvature of a complete minimal surface and the modified defect relation of a Fermat hypersurface for the Gauss map} 

\author{Si Duc Quang}
\author{Nguyen Thi Quynh Chi}
\address{Department of Mathematics, Hanoi National University of Education\\
136-Xuan Thuy, Cau Giay, Hanoi, Vietnam}
\email{quangsd@hnue.edu.vn,	ntqchi2004@gmail.com}

\begin{document}

\begin{abstract}
In this paper, we establish some modified defect relations for the Gauss map $g$ of a complete minimal surface $S\subset\mathbb R^m$ into $\mathbb P^n(\mathbb C)\ (n=m-1)$ with only a single Fermat hypersurface $Q$ of $\mathbb P^n(\mathbb C)$.  In particular, we show that $S$ must have finite total curvature if the image $g(S)$ intersects $Q$ with only a finite number of times and the degree of $Q$ is sufficiently large.
\end{abstract}

\maketitle

\def\thefootnote{\empty}
\footnotetext{
2010 Mathematics Subject Classification:
Primary 53A10, 53C42; Secondary 30D35, 32H30.\\
\hskip8pt Key words and phrases: Gauss map, value distribution, holomorphic curve, modified defect relation, ramification, hypersurface, total curvature.}


\section{Introduction and Main results} 

Let $f$ be a nonconstant holomorphic curve from $\C$ into $\P^n(\C)$ and $H$ be a hyperplane in $\P^n(\C)$. The defect of $H$ for $f$ is defined by 
$$\delta^k_f(H)=1-\underset{r\longrightarrow +\infty}{\mathrm{lim\ sup}}\frac{N^{[k]}_f(r,H_i)}{T_f(r)},$$  
where $T_f(r)$ is the characteristic function of $f$ and $N^{[k]}_f(r,H)$ is the truncated (to level $k$) counting function of $f^{-1}(H)$ (see \cite{NO} for the definitions). We see that if $f$ intersects $H$ with only a finite number of times then $\delta^k_f(H)=1$.

Let $\{H_i\}_{i=1}^q$ be a family of hyperplanes of $\P^n(\C)$ in general positon, i.e., $\bigcap_{0\le j\le n}H_{i_j}=\varnothing$ for any $1\le i_0<\cdots<i_n\le q$. From the second main theorem in Nevanlinna theory (due to H. Cartan \cite{Ca} and E. Nochka \cite{Noc83}) we get the defect relation that 
$$\sum_{i=1}^q\delta^k_f(H_i)\le 2n-k+1$$ 
if the image of $f$ is contained in a linear subspace of $k$-dimension of $\P^n(\C)$ but not in any subspace with dimension less than $k$. 
 Therefore, if $f$ intersects $q$ hyperplanes $\{H_i\}_{i=1}^q$ with only a finite number of times then $f$ is linearly degenerate for the case $q\ge n+2$, and $f$ must be constant for the case $q\ge n+2.$

In the late 1990s, H. Fujimoto extended these results to the Gauss maps of complete minimal surfaces by introducing the concept of the modified defect for holomorphic curves on minimal Riemann surfaces in $\R^n$ (see \cite{Fu89, Fu90, Fu91}), which serves as an analogue to the classical defect relation in Nevanlinna theory. Fujimoto established several modified defect relations for the Gauss maps of minimal surfaces and used them to investigate various properties of these maps. For instance, he showed that if the Gauss map of a nonflat complete minimal surface immersed in $\R^{n+1}$, whose total curvature is infinite, intersects $q$ hyperplanes in general position of $\P^{n}(\C)$ with only a finite number of times then $q\le n+1+\frac{n(n+1)}{2}$ (see \cite[Corollary 1.4]{Fu91}). Subsequently, many authors have extended the results of Fujimoto to the case where the Gauss maps are ramified over families of hyperplanes, meaning the multiplicity of the intersection points between the Gauss maps and the hyperplanes exceeds a certain threshold. For further details, we refer the reader to the papers \cite{Ha18a,Ha18b,HTP,Ru93} and the references therein. Especially, recently the first author \cite{Q24} generalized all these results to the case of the Gauss maps and the families hypersurfaces in subgeneral position. In this paper, we will study the Gauss maps of complete minimal surfaces under the condition of intersection with only a single Fermat hypersurface. Under this condition, we will derive more refined estimates compared to the results in the previous paper. In order to state the main results, we first recall the following essential notions (see \cite{Fu89,Fu90,Fu91} and \cite{Q24}).

 Let $f$ be a holomorphic curve from a Riemann surface $S$ into $\P^n(\C)$ and $Q$ a hypersurface in $\P^n(\C)$ of degree $d$. By $\nu_{Q(f)}$ we denote the pullback divisor $Q$ by $f$. 
Let $K$ be a compact subset of $S$ and $A=S\setminus K$. 

The $S$-defect truncated to level $k$ of $Q$ for $f$ (over $A$) is defined by 
$$\delta^{S,k}_{f,A}(Q):=1- \mathrm{inf}\{\eta\ |\ \eta \text{ satisfies the condition }(*)_S\}.$$
Here, the condition $(*)_S$ means that there exists a $[-\infty,\infty)$-valued continuous subharmonic function $u\ (\not\equiv -\infty)$ on $A$ satisfying the following conditions:
\begin{itemize}
\item[$(D_1)$] $e^u\le \|F\|^{d\eta}$,
\item [$(D_2)$] for each $\xi\in f^{-1}(Q)$ there exists the limit 
$$\underset{z\rightarrow\xi}\lim (u(z)-\min\{\nu_{Q(f)}(z),k\}\log|z -\xi|) \in[-\infty,\infty),$$
where $z$ is a holomorphic local coordinate around $\xi$.
\end{itemize}

The $H$-defect truncated to level $k$ of $Q$ for $f$ (over $A$) is defined by
$$\delta^{H,k}_{f,A}(Q):=1- \mathrm{inf}\{\eta\ |\ \eta \text{ satisfies the condition }(*)_H\}.$$
Here, the condition $(*)_H$ means that there exists a $[-\infty,\infty)$-valued continuous subharmonic function $u\ (\not\equiv\infty)$ on $A$, which is harmonic on $A\setminus f^{-1}(Q)$ and satisfies conditions $(D_1),(D_2)$.

Let $u$ be a complex values function on a domain $D\subset S$. The function $u$ is said to have mild singularities if it is of $\mathcal C^\infty$ on $D$ outside a discrete subset $E$ and every point $a\in D$ has a neighborhood $U$ such that for a holomorphic local coordinate $z$ with $z(a)=0$ on $U$ we can write
$$ |u(z)|=|z|^\alpha u^*(z)\prod_{j=1}^q\left|\log\dfrac{1}{|g_j(z)|v_j(z)}\right|^{r_j}$$
on $U\setminus E$ with some real number $\sigma$, non-positive real numbers $r_j$, nonzero holomorphic functions $g_j$ with $g_j(0)=0$ and some positive $\mathcal C^{\infty}$ functions $u^*$ and $v_j$ on $U$, where $0\le q<\infty$.

The modified $H$-defect truncated to level $k$ of $Q$ for $f$  is defined by
$$D^k_{f}(Q) := 1 - \mathrm{inf}\{1-\eta\ |\ \eta \text{ satisfies the condition (*)}.\}$$
Here, the condition (*) means that: there are a compact subset $K$ of $S$, a divisor $\nu$ and a continuous real-valued bounded function $k$ with mild singularities on $S\setminus K$, a positive constant $c$ such that $\nu(z)>c$ for each $z\in\supp (\nu)$ and
$$ [\min\{m,\nu_{Q(f)}\}]+[\nu]=d\eta f^*\Omega+dd^c[\log|k|^2] $$
in the sense of current, where $\Omega$ is the Fubini-Study form of $\P^n(\C)$. 


By \cite[Proposition 2.3]{Fu90} and \cite[Remark 5.3]{Fu91}, one has
\begin{align*}
0 \le  \delta^{H,k}_{f,A}(Q) \le  \delta^{S,k}_{f,A}(Q)\le  1\text{ and }0 \le \delta^{H,k}_{f,A}(Q)\le D^{k}_{f}(Q)\le 1.
\end{align*}

Let $Q_1,...,Q_q\ (q\ge n+1)$ be $q$ hypersurfaces in $\P^n(\C)$. The family $\{Q_i\}_{i=1}^q$ is said to be in general position if 
$$ \bigcap_{j=1}^{n+1}Q_{i_j}=\varnothing \ \forall\ 1\le i_1<\cdots <i_{n+1}.$$
Throughout this paper, sometimes we will denote by the same notation $Q$ for the defining homogeneous polynomial of the hypersurface $Q$ if there is no confusion.

Our first main result is stated as follows.
 
\begin{theorem}\label{1.1} 
Let $S$ be a complete minimal surface in $\R^{n+1}\ (n\ge 2)$ with the Gauss map $g$, which has a reduced representation $G=(g_0,\ldots,g_n)$. Let $Q_0,\ldots,Q_n$ be $n+1$ hypersurfaces of $\P^n(\C)$ of the same degree $\ell$ in general position and let $Q$ be a Fermat hypersurface of $\P^n(\C)$ defined by the form
$$Q=\sum_{j=0}^na_jQ_j^d,$$
where $a_0,\ldots,a_n$ are nonzero constants.
Assume that the map $(Q_0^d(G),\ldots,Q_n^d(G))$ is linearly nondegenerate and 
$$D^{n}_{g}(Q)>\frac{n(n+1)}{d}+\frac{n(n+1)}{2\ell d}.$$
Then $S$ has finite total curvature.
\end{theorem}
Since $\delta^{H,n}_{g,A}(Q)\le D^{n}_{g}(Q)$, from Theorem \ref{1.1} we have the following corollary.
\begin{corollary}\label{1.2} 
Let $S,\{Q_i\}_{i=0}^n,Q,g$ be as in Theorem \ref{1.1}, $K$ a compact subset of $S$ and $A=S\setminus K$. Assume that the map $(Q_0^d(G),\ldots,Q_n^d(G))$ is linearly nondegenerate and
$$\delta^{H,n}_{g,A}(Q)>\frac{n(n+1)}{d}+\frac{n(n+1)}{2\ell d}.$$
Then $S$ has finite total curvature.
\end{corollary}

Note that, if $g$ is ramified over a hypersurface $Q$ with multiplicity at least $m$ on $A$ then $\delta^{H,n}_{g,A}(Q)\ge 1-\dfrac{n}{m}$. Hence, from Corollary \ref{1.2}, we immediately get the following corollary on the ramification of the Gauss map.
\begin{corollary}\label{1.4}
Let $S,\{Q_i\}_{i=0}^n,Q,g$ be as in Theorem \ref{1.1}. Assume that $g$ is ramified over $Q$ with multiplicity at least $m$ on $A$ and
$$1-\frac{n}{m}>\frac{n(n+1)}{d}+\frac{n(n+1)}{2\ell d}.$$
Then $S$ has finite total curvature.

In particular, if $g$ intersects only a finite number of times the hypersurface $Q$ and 
$$d>n(n+1)\left(1+\frac{1}{2\ell}\right),$$
then $S$ must have finite total curvature.
\end{corollary}

Our second purpose in this paper is to give an $S$-defect relation as follows. 
\begin{theorem}\label{1.3} 
Let $S$ be a complete minimal surface with finite total curvature in $\R^{n+1}\ (n\ge 2)$ and $\{Q_i\}_{i=0}^n,Q,g$ as in Theorem \ref{1.1}. Assume that the map $(Q_0^d(G),\ldots,Q_n^d(G))$ is linearly nondegenerate. Then we have
$$\delta^{S,n}_{g,S}(Q)\le\frac{n(n+1)}{d}+\frac{n(n+1)}{2\ell d}.$$
\end{theorem}
 
Since $ \delta^{H,n}_{f,S}(Q) \le  \delta^{S,n}_{f,S}(Q)$, then by combining Corollary \ref{1.2} and Theorem \ref{1.3}, we also immediately get the following $H$-defect relation.
\begin{corollary}\label{1.6}
Let $S,\{Q_i\}_{i=0}^n,Q$ and $g$ be as in Theorem \ref{1.1}. Assume that the map $(Q_0^d(G),\ldots,Q_n^d(G))$ is linearly nondegenerate. Then we have
$$\delta^{H,n}_{g,S}(Q)\le \frac{n(n+1)}{d}+\frac{n(n+1)}{2\ell d}.$$
\end{corollary}

\section{Defect relation of holomorphic curves from a punctured disk into a projective variety with a family of hypersurfaces}

\noindent
{\bf A. Nevanlinna functions.}\ For each $s>0$, we define the punctured disk
$$\Delta_{s,\infty}=\{z\in\C\ |\ s\le |z|<\infty\}.$$
For a divisor $\nu$ on an open neighborhood of $\Delta_{s,\infty}$ and for a positive integer $p$ or $p= \infty$, we set $n^{[p]}(t,s) =\sum\limits_{s\le |z|\leq t}\min\{p,\nu (z)\}$ and define the counting function of $\nu$ by
$$N^{[p]}(r,s,\nu)=\int\limits_{s}^r \dfrac {n^{[p]}(t,s)}{t}dt \quad (s<r<\infty).$$
Let $\varphi$ be a meromorphic function on $\Delta_{s,\infty}$. We will write $N_{\varphi}^{[p]}(r,s)$ for $N^{[p]}(r,s,\nu^0_{\varphi}),$ where  $\nu^0_\varphi$ is the zero divisor of $\varphi$. For brevity, we will omit the character $^{[p]}$ if $p=\infty$. 

The proximity function of $\varphi$ (with respect to the point $\infty$) is defined by
$$ m(r,s,\varphi)=\int\limits_{0}^{2\pi}\log^+|\varphi(re^{i\theta})|\dfrac{d\theta}{2\pi}-\int\limits_{0}^{2\pi}\log^+|\varphi(se^{i\theta})|\dfrac{d\theta}{2\pi}\ (s<r<\infty),$$
where $\log^+x=\log(\max\{0,x\})$ for a real number $x$. The Nevanlinna's characteristic function of $\varphi$ is defined by
$$ T(r,s,\varphi):=m(r,s,\varphi)+N_{1/\varphi}(r,s).$$

Throughout this paper, we fix a homogeneous coordinates system $(x_0:\cdots :x_n)$ on $\P^n(\C)$. Let $f : \Delta_{s,\infty} \longrightarrow \P^n(\C)$ be a holomorphic curve with a reduced representation $F = (f_0, \ldots,f_n)$ on $\Delta_{s,\infty}$. We set 
$$\|F\|=(|f_0|^2+\cdots+|f_n|^2)^{1/2}.$$
The characteristic function of $f$ is defined by 
$$ T_f(r,s)=\int_{0}^{2\pi}\log\|F(re^{i\theta})\|\dfrac{d\theta}{2\pi}-\int_{0}^{2\pi}\log\|F(se^{i\theta})\|\dfrac{d\theta}{2\pi}\ (s<r<R). $$
For a meromorphic function $\varphi$, which is regarded as a holomorphic curve into $\P^1(\C)$, then
$$ T_\varphi(r,s)=T(r,s,\varphi)+O(1).$$

\begin{theorem}\label{thm2.7} 
Let $\{Q_i\}_{i=0}^n$ be $n+1$ hypersurfaces of $\P^n(\C)$ in general position of the same degree $p$ and $Q$ a hypersurface given by $Q=\sum_{i=0}^na_iQ_i^d$, where $a_i\ (0\le i\le n)$ are nonzero constants and $d$ is a positive integer. Let $f$ be a holomorphic mapping of $\Delta_{s,\infty}$ into $\P^n(\C)$ with a reduced representation $F=(f_0,\ldots,f_n)$ such that $Q_0^d(F),\ldots,Q_n^d(F)$ are linearly independent. Then, we have
$$ \biggl \|\ \left(1-\frac{n(n+1)}{d}\right)T_f(r,s)\le \frac{1}{pd}N^{[n]}_{Q(f)}(r,s)+o(T_f(r,s))+O(\log r).$$
\end{theorem}
\begin{proof}
Let $\tilde f$ be the holomorphic curve with the reduced representation $(Q_0^d(F),\ldots,Q_n^d(F))$. Then $\tilde f$ is linearly nondegenerate. By the second main theorem for hyperplanes, we have
\begin{align*}
\bigl\|\ T_{\tilde f}(r,s)&\le\sum_{i=0}^nN^{[n]}_{Q_i^d(f)}(r,s)+N^{[n]}_{Q(f)}(r,s)+o(T_{\tilde f}(r,s))+O(\log r)\\
&\le n\sum_{i=0}^nN_{Q_i(f)}(r,s)+N^{[n]}_{Q(f)}(r,s)+o(T_{\tilde f}(r,s))+O(\log r)\\
&\le n(n+1)pT_f(r,s)+N^{[n]}_{Q(f)}(r,s)+o(T_{\tilde f}(r,s))+O(\log r).
\end{align*}
Since $T_{\tilde f}(r)=dpT_f(r,s)+O(1)$, the above inequality implies that
$$ \biggl \|\ \left(1-\frac{n(n+1)}{d}\right)T_f(r,s)\le \frac{1}{pd}N^{[n]}_{Q(f)}(r,s)+o(T_f(r,s))+O(\log r).$$
The theorem is proved.
\end{proof}

For a hypersurface $Q$ of degree $\ell$, if $f$ has an essential singularity at $\infty$ then the truncated defect of $Q$ for $f$ is defined by
$$ \delta^{[n]}_f(Q)=1-\underset{r\longrightarrow\infty}{\rm limsup}\dfrac{N_{Q(f)}^{[n]}(r,s)}{\ell dT_f(r,s)}.$$
As we known that
$$\delta^{H,n}_{f,\Delta_{s,\infty}}(Q)\le\delta_{f,\Delta_{s,\infty}}^{S,n}(Q)\le\delta_f^{[n]}(Q)\le 1.$$
Then, from Theorem \ref{thm2.7} we have the following corollary.
\begin{corollary}\label{cor2.9}
With the above notation, if $f$ has an essential singularity at $\infty$ then
$$\delta^{H,n}_{f,\Delta_{s,\infty}}(Q)\le\delta^{[n]}_f(Q)\le \dfrac{n(n+1)}{d}.$$
\end{corollary}

\section{Derived curves and a result of the generalized Schwarz lemma}

Let $S$ be an open Riemann surface and $ds^2$ a pseudo-metric on $S$ which is locally written as $ds^2=\lambda^2|dz|^2$, where $\lambda$ is a nonnegative real-value function with mild singularities and $z$ is a holomorphic local coordinate. The divisor of $ds^2$ is defined by $\nu_{ds}:=\nu_\lambda$ for each local expression $ds^2=\lambda^2|dz|^2$, which is globally well-defined on $S$. 
The Ricci of  $ds^2$ is defined by 
$$ \mathrm{Ric}_{ds^2}=-dd^c\log\lambda^2$$
for each local expression  $ds^2=\lambda^2|dz|^2$. This definition is also globally well-defined on $S$.

We say that $ds^2$ is a continuous pseudo-metric if $\nu_{ds}\ge 0$ everywhere. We say that $ds^2$ have strictly negative curvature on $S$ if there is a positive constant $C$ such that
$$ -\mathrm{Ric}_{ds^2}\ge C\Omega_{ds^2}, $$
where $\Omega_{ds^2}=\lambda^2\cdot\dfrac{\sqrt{-1}}{2}\cdot dz\wedge d\bar z$.

Let $f$ be a linearly nondegenerate holomorphic map of $S$ into $\P^n(\C)$ and $F=(f_0,\ldots,f_n)$ a reduced representation of $f$. 
Consider the holomorphic map
$$F_{p} = (F_{p})_z := F^{(0)} \wedge F^{(1)} \wedge \cdots \wedge F^{(p)} : S\rightarrow \bigwedge_{p+1}\C^{n+1}$$
for $0\le p\le M$, where 
\begin{itemize}
\item $F^{(0)}:=F=(f_0,\ldots,f_n)$,
\item $F^{(l)}=(F^{(l)})_z:=\left ((f_0)^{(l)}_z,\ldots,(f_n)^{(l)}_z\right)$ for each $l=0, 1,\ldots , p$,
\item $(f_i)^{(l)}_z \ (i =0,\ldots, n)$ is the $l^{th}$- derivatives of $f_i$ taken with respect to $z$.
\end{itemize}
The norm of $F_{p}$ is given by
$$|F_{p}|:=\left (\sum_{0\le i_0<i_1<\cdots<i_p\le n}\left |W_z(f_{i_0},\ldots,f_{i_p})\right|^2\right)^{1/2}, $$
where $W_z(f_{i_0},\ldots,f_{i_p}):=\det\left ((f_{i_j})^{(l)}_z\right)_{0\le l,j\le p}$.  

Now, for a hyperplane $H$ in $\P^n(\C)$ defined by a linear form (denoted again by $H$) á follows
$$H=\sum_{i=0}^na_ix_i,$$
where $(x_0:\cdots:x_n)$ is homogeneous coordinates system on $\P^n(\C)$. We define
\begin{align*}
F_{0}(H)&=a_0f_0+\cdots+a_nf_n=H(F),\\ 
|F_{p}(H)|&=\left (\sum_{0\le i_1<\cdots<i_p\le n}\sum_{\ell\ne i_1,\ldots,i_p}\left |a_lW_z(f_\ell,f_{i_1},\ldots,f_{i_p})\right|^2\right)^{1/2}.
\end{align*}
Finally, for $0\le p\le n$, the $p^{th}$-contact function of $f$ for $H$ is defined (not depend on the choice of the local coordinate $z$) by
$$\varphi_{p}(H):=\dfrac{|F_{p}(H)|^2}{|F_{p}|^2}.$$

For each $p\ (0\le p\le n-1)$, we set $n_p=\binom{n+1}{p+1}-1$. Denote by $\pi_p$ the canonical projection from $\bigwedge^{p+1}\C^{n+1}\sim\C^{n_p+1}$ onto $\P^{n_p}(\C)$ and by $\Omega_p$ the pullback of the Fubini-Study form on $\P^{n_p}(\C)$ by the map $\pi\circ F_{p}$, i.e., $\Omega_p = dd^c\log |F_{p}|^2$. 

\begin{theorem}[{see \cite[Theorem 2.5.3]{Fu90}}]\label{thm3.3}
Let $H_1, \cdots, H_q$ be hyperplanes in $\mathbb{P}^n(\mathbb{C})$ in general position. For every $\epsilon>0$ there exist positive numbers $\delta(>1)$ and $C$, depending only on $\epsilon$ and $H_j, 1 \leq j \leq q$, such that
\begin{align*}
&dd^c \log \frac{\Pi_{p=0}^{n-1}\left|F_p\right|^{2 \epsilon}}{\Pi_{1 \leq j \leq q, 0 \leq p \leq n-1} \log ^{2}\left(\delta / \varphi_p\left(H_j\right)\right)}\\
&\geq C\left(\frac{\left|F_0\right|^{2(q-n-1)}\left|F_n\right|^2}{\Pi_{j=1}^q\left(\left|H_j\left(F\right)\right|^2 \Pi_{p=0}^{n-1} \log ^2\left(\delta / \varphi_p\left(H_j\right)\right)\right)}\right)^{\frac{2}{n(n+1)}} dd^c|z|^2
\end{align*}
\end{theorem}

\begin{theorem}[{cf. \cite[Proposition 2.5.7]{Fu93}}]\label{thm3.4}
Set $\sigma_p=p(p+1)/2$ for $0\le p\le n+1$ and $\tau_m=\sum_{p=1}^m\sigma_p$. Then, we have
$$ dd^c\log(|F_{0}|^2\cdots |F_{n-1}|^2)\ge\dfrac{\tau_n}{\sigma_n}\left(\dfrac{|F_{0}|^2\cdots |F_{n}|^2}{|F_{0}|^{2\sigma_{n+1}}}\right)^{1/\tau_n}dd^c|z|^2. $$
\end{theorem}

\begin{theorem}\label{thm3.5}
With the notations and the assumption in Theorem \ref{thm3.3}, we have
$$ \nu_{\phi}+\sum_{j=1}^q\omega_j\cdot\min\{\nu_{H_j(f)},n\}\ge 0 $$
where $\phi=\dfrac{|F_n|}{\prod_{j=1}^q|H_j(F)|^{\omega_j}}$ and $\nu_\phi$ is the divisor of the function $\phi$.
\end{theorem}

\begin{lemma}[{Generalized Schwarz's Lemma \cite{A38}}]\label{lem3.6}  
Let $v$ be a non-negative real-valued continuous subharmonic function on $\Delta (R)=\{z\in\C\ |\ \|z\|<R\}$. If $v$ satisfies the inequality $\Delta\log v\ge v^2$ in the sense of distribution, then
$$v(z) \le \dfrac{2R}{R^2-|z|^2}.$$
\end{lemma}

We prove the following lemma on the construction of a metric with strictly negative curvature from the derive curve of a holomorphic curve and a Fermat hypersurface.
\begin{lemma}\label{lem3.7} 
Let $Q_0,\ldots,Q_n$ be hypersurfaces of $\P^n(\C)$ in general position of the same degree $\ell$ and $Q$ a hypersurface of $\P^n(\C)$ defined by
$$Q=\sum_{j=0}^na_{j}Q_j^d,$$
where $a_i$'s are nonzero constants. Let $g:\Delta (R)\rightarrow \P^n(\C)$ be a holomorphic map with a reduced representation $(g_0,\ldots,g_n)$ such that the map $f:\Delta (R)\rightarrow \P^n(\C)$, which has the reduced representation $F=(Q_0^d(g),\ldots,Q_n^d(g))$, is linearly nondegenerate. Let $H_j$ is the linear form given by $H_j(x_0,\ldots,x_n)=x_j\ (0\le j\le n)$ and $H_{n+1}$ is the linear form given by $H_{n+1}(x_0,\ldots,x_n)=\sum_{j=0}^na_jx_j$. Assume that there is a positive real number $\theta\in (0,1)$, a divisor $\nu$, a continuous real-valued bounded function $k$ with mild singularities on $S$, a positive constant $c$ such that $\nu(z)>c$ for each $z\in\supp \nu$ and
$$ [\min\{n,\nu_{H_{n+1}(F)}\}]+[\nu]=d\ell\theta \Omega_g+dd^c[\log|k|^2],$$
where $\Omega_g$ is the pull-back of the Fubini Study form $\Omega_n$ on $\P^n(\C)$ by $g$. Let $h=|k|\cdot \|F\|^\theta$. Then for an arbitrarily given $\epsilon$ satisfying 
$$\gamma=1-\theta-\frac{n(n+1)}{d}>\epsilon\left(\sigma_{n+1}+\theta\right),$$ 
the pseudo-metric $d\tau^2=\eta^2|dz|^2$, where
$$ \eta=\left(\dfrac{|F_{0}|^{\gamma-\epsilon\left(\sigma_{n+1}+\theta\right)}h^{1+\epsilon}|F_n|\prod_{p=0}^{n}|F_{p}|^{\epsilon}}{|H_{n+1}(F)|\cdot\prod_{j=0}^{n}|H_j(F)|^{1-\frac{n}{d}}\cdot\prod_{j=0}^{n+1}\prod_{p=0}^{n-1}\log (\delta/\varphi_{p}(H_j))}\right )^{\frac{1}{\sigma_n+\epsilon\tau_n}}$$
and $\delta$ is the number satisfying the conclusion of Theorem \ref{thm3.3}, is continuous and has strictly negative curvature.
\end{lemma}
\begin{proof}
Without loss of generality, assume that $k\le 1$, and hence $h\le \|F\|^\theta$. It is obvious that the function $\eta$ is continuous at every point $z$ with $\prod_{j=0}^nH_j(F)(z)\ne 0$. Now, for $z_0\in S$ with $\prod_{j=0}^nH_j(F)(z_0)= 0$, we have
\begin{align}\label{new1}
\nu_{h}(z_0)\ge \min\{n,\nu_{H_{n+1}(F)}(z_0)\}.
\end{align}
Also, by Theorem \ref{thm3.5}, one has
\begin{align}\label{new4}
\begin{split}
\nu_{F_n}(z_0)&-\nu_{H_{n+1}(F)}(z_0)-(1-\frac{n}{d})\sum_{j=0}^n\nu_{H_j(F)}(z_0)\\
&\ge -\sum_{j=0}^{n+1}\min\{n,\nu_{H_j(F)}(z_0)\}+\frac{n}{d}\sum_{j=0}^n\nu_{H_j(F)}(z_0)\\
&\ge  -\min\{n,\nu_{H_{n+1}(F)}(z_0)\}.
\end{split}
\end{align}
Combining the inequalities (\ref{new1}), (\ref{new4}) and the definition of $\eta$, we have
\begin{align*}
\nu_{\eta}(z_0)&\ge \frac{1}{\sigma_n+\epsilon\tau_n}\biggl (\nu_{F_n}(z_0)+(1+\epsilon)\min\{\nu_{H_{n+1}(F)}(z_0),n\}\\
&-\nu_{H_{n+1}(F)}(z_0)-(1-\frac{n}{d})\sum_{j=0}^n\nu_{H_j(F)}(z_0)\biggl)\\
&\ge\frac{\epsilon}{\sigma_n+\epsilon\tau_n}\min\{n,\nu_{H_{n+1}(F)}(z_0)\}\ge 0.
\end{align*}
Then $d\tau^2$ is continuous pseudo-metric on $\Delta(R)$. 

We will show that $d\tau^2$ has strictly negative curvature on $\Delta$. By the inequality (\ref{new4}) and the definition of  $h$, we have
\begin{align*}
dd^c&\log\dfrac{|F_n|}{|H_{n+1}(F)|\cdot\prod_{j=0}^{n}|H_j(F)|^{1-\frac{n}{d}}}+(1+\epsilon)dd^c\log h\\
&\ge -\frac{1}{2}[\min\{n,\nu_{H_{n+1}(F)}\}]+\frac{1}{2}(1+\epsilon)\left([\min\{n,\nu_{H_{n+1}(F)}\}]+[\nu]\right)\ge 0.
\end{align*}
Then by Theorems \ref{thm3.3}, \ref{thm3.4}, with the fact that $dd^c\log|F_n|=0$ (because $F_n$ is a holomorphic function) and the definition of $\eta$, we have
\begin{align}\label{new3}
\begin{split}
&dd^c\log\eta^2\ge\dfrac{\gamma-\epsilon\left(\sigma_{n+1}+\theta\right)}{\sigma_n+\epsilon\tau_n}dd^c\log|F_0|^2+\dfrac{\epsilon}{2(\sigma_n+\epsilon\tau_n)}dd^c\log\left(|F_{0}|^2\cdots|F_{n-1}|^2\right)\\
&\quad\quad\quad\quad +\dfrac{1}{\sigma_n+\epsilon\tau_n}dd^c\log\dfrac{\prod_{p=0}^{n-1}|F_{p}|^{2(\frac{\epsilon}{2})}}{\prod_{j=0}^{n+1}\prod_{p=0}^{n-1}\log^{2}(\delta/\varphi_{p}(H_j))}\\
&\ge\dfrac{\epsilon\tau_n}{2\sigma_n(\sigma_n+\epsilon\tau_n)}\left(\dfrac{|F_0|^2\cdots |F_n|^2}{|F_0|^{2\sigma_{n+1}}}\right)^{\frac{1}{\tau_n}}dd^c|z|^2\\
&+C_0\left (\dfrac{|F_0|^{2}|F_n|^2}{\prod_{j=0}^{n+1}(|H_j(F)|^2\prod_{p=0}^{n-1}\log^2(\delta/\varphi_{p}(H_j)))}\right)^{\frac{1}{\sigma_n}}dd^c|z|^2\\
&\ge\min \{\dfrac{1}{2\sigma_n(\sigma_n+\epsilon\tau_n},\dfrac{C_0}{\sigma_n}\}\epsilon\tau_n\left(\dfrac{|F_0|^2\cdots |F_n|^2}{|F_0|^{2\sigma_{n+1}}}\right)^\frac{1}{\tau_n}dd^c|z|^2\\
&+\min \{\dfrac{1}{2\sigma_n(\sigma_n+\epsilon\tau_n},\dfrac{C_0}{\sigma_n}\}\sigma_n\left (\dfrac{|F_0|^{2}|F_n|^2}{\prod_{j=0}^{n+1}(|H_j(F)|^2\prod_{p=0}^{n-1}\log^2(\delta/\varphi_{p}(H_j)))}\right)^{\frac{1}{\sigma_n}}dd^c|z|^2\\
&\ge C_1\left (\dfrac{|F_{0}|^{1-\epsilon\sigma_{n+1}}|F_n|\prod_{p=0}^n|F_{p}|^\epsilon}{\prod_{j=0}^{n+1}(|H_j(F)|\prod_{p=0}^{n-1}\log(\delta/\varphi_{p}(H_j)))}\right)^{\frac{2}{\sigma_n+\epsilon\tau_n}}dd^c|z|^2
\end{split}
\end{align}
for some positive constants $C_0,C_1$, where the last inequality comes from the H\"{o}lder's inequality.

On the other hand, we have $|h|\le \|F\|^{\theta}\ |H_j(F)|\le \|F_0\|,(0\le j\le n)$, and hence
\begin{align*}
|F_{0}|^{1-\epsilon\sigma_{n+1}}&=|F_{0}|^{\gamma-\epsilon\sigma_{n+1}+\theta+\frac{n(n+1)}{d}}\\
&=|F_{0}|^{\gamma-\epsilon(\sigma_{n+1}+\theta+\frac{n(n+1)}{d})+(1+\epsilon)\theta+(1+\epsilon)\frac{n(n+1)}{d}}\\
&\ge |F_{0}|^{\gamma-\epsilon(\sigma_{n+1}+\theta+\frac{n(n+1)}{d})}|h|^{1+\epsilon}\cdot\prod_{j=0}^n|H_j(F)|^{\frac{n}{d}}.
\end{align*}
This yields that $\Delta \log\eta^2\ge C_2\eta^2$ for a positive constant $C_2$, and hence $d\tau^2$ has strictly negative curvature.
\end{proof}

Applying Lemma \ref{lem3.7}, we will prove the following main lemma of this paper.
\begin{lemma}\label{ML}
Let the notation be as in Lemma \ref{lem3.7}. Then for every $\epsilon >0$ satisfying 
$$\gamma=1-\theta-\frac{n(n+1)}{d}>\epsilon\left(\sigma_{n+1}+\theta\right),$$ 
there exists a positive constant $C$, depending only on $ Q_j (0\le j\le n)$ and $\theta$, such that
\begin{align*}
\dfrac{|F_{0}|^{\gamma-\epsilon(\sigma_{n+1}+\theta)}h^{1+\epsilon}|F_n^{1+\epsilon}|\prod_{j=0}^{n+1}\prod_{p=0}^{n-1}|F_{p}(H_j)|^{\epsilon/{n+2}}}{|H_{n+1}(F)|.\prod_{j=0}^{n}|H_j(F)|^{1-\frac{n}{d}}}\le C\left(\dfrac{2R}{R^2-|z|^2}\right)^{\sigma_n+\epsilon\tau_n}.
\end{align*}
\end{lemma}
\begin{proof}
Similarly as in the proof of Lemma \ref{lem3.7}, we have
$$dd^c\log\eta\le C_2\eta^2dd^c|z|^2.$$
According to Lemma \ref{lem3.6}, one has $ \eta\le C_3\dfrac{2R}{R^2-|z|^2}$ for a positive constant $C_3$, and hence
$$ \left( \dfrac{|F_{0}|^{\gamma-\epsilon(\sigma_{n+1}+\theta)}h^{1+\epsilon}|F_n|\prod_{p=0}^{n}|F_{p}|^{\epsilon}}{|H_{n+1}(F)|\cdot\prod_{j=0}^{n}|H_j(F)|^{1-\frac{n}{d}}\cdot\prod_{j=0}^{n+1}\prod_{p=0}^{n-1}\log (\delta/\varphi_{p}(H_j)))}\right )^{\frac{1}{\sigma_n+\epsilon\tau_n}}\le C_3\dfrac{2R}{R^2-|z|^2}.$$
It follows that
$$ \left( \dfrac{|F_{0}|^{\gamma-\epsilon(\sigma_{n+1}+\theta)}h^{1+\epsilon}|F_n|^{1+\epsilon}\prod_{j=0}^{n+1}\prod_{p=0}^{n-1}|F_{p}(H_j)|^{\epsilon/(n+2)}}{|H_{n+1}(F)|\cdot\prod_{j=0}^{n}|H_j(F)|^{1-\frac{n}{d}}\cdot\prod_{j=0}^{n+1}\prod_{p=0}^{n-1}(\varphi_{p}(H_j))^{\epsilon/2(n+2)}(\log (\delta/\varphi_{p}(H_j)))}\right )^{\frac{1}{\sigma_n+\epsilon\tau_n}}$$
$$\le C_3\dfrac{2R}{R^2-|z|^2}.$$
Since the function $x^{\frac{\epsilon}{n+2}}\log\left (\dfrac{\delta}{x^2}\right)\ (0<x\le 1)$ is bounded, we have
$$\left( \dfrac{|F_{0}|^{\gamma-\epsilon(\sigma_{n+1}+\theta)}h^{1+\epsilon}|F_n|^{1+\epsilon}\prod_{j=0}^{n+1}\prod_{p=0}^{n-1}|F_{p}(H_j)|^{\epsilon/(n+2)}}{|H_{n+1}(F)|\cdot\prod_{j=0}^{n}|H_j(F)|^{1-\frac{n}{d}}}\right )^{\frac{1}{\sigma_n+\epsilon\tau_n}}\le C_4\dfrac{2R}{R^2-|z|^2},$$
for a positive constant $C_4$. The lemma is proved.
\end{proof}

\section{Modified defect relation for Gauss map with hypersurfaces}

In this section, we will prove the main theorems of the paper. Firstly, we need the following essential lemma.

\begin{lemma}[{cf. \cite[Lemma 1.6.7]{Fu93}}]\label{lem4.1}
Let $d\sigma^2$ be a conformal flat metric on an open Riemann surface $S$. Then for every point $p\in S$, there is a holomorphic and locally biholomorphic map $\Phi$ of a disk $\Delta (R_0):=\{w:|w| <R_0\}\ (0 <R_0\le\infty)$ onto an open neighborhood of $p$ with $\Phi(0)=p$ such that $\Phi$ is a local isometry, namely the pull-back $\Phi^*(d\sigma^2)$ is equal to the standard (flat) metric on $\Delta(R_0)$, and for some point $a_0$ with $|a_0|=1$, the curve $\Phi (\overline{0,R_0a_0})$ is divergent in $S$ (i.e. for any compact set $K\subset S$, there exists an $s_0<R_0$ such that $\Phi (\overline{0,s_0a_0})$ does not intersect $K$).
\end{lemma}

Let $x=(x_0,\ldots,x_{n}): S\rightarrow\R^{n+1}$ be the immersion of a minimal surface $S$ into $\R^{n+1}$. Let $(u,v)$ be a local isothermal coordinate of $S$. Then $z=u+iv$ is a local holomorphic coordinate of $S$. The generalized Gauss map $g$ of $x$ is defined (locally) by
$$ g:S\rightarrow\P^{n}(\C), g:=\left(\dfrac{\partial x_0}{\partial z}:\cdots:\dfrac{\partial x_{n}}{\partial z}\right).$$
We set 
$$G_z=\left(\dfrac{\partial x_0}{\partial z},\ldots,\dfrac{\partial x_{n}}{\partial z}\right).$$
Then, $G_z$ is a local reduced representation of $g$. If $\xi$ is another local holomorphic, then 
$$G_\xi=G_z\cdot\left (\dfrac{dz}{d\xi}\right)$$
(this yields that the map $g$ is global well-defined). Also, the metric $ds^2$ on $S$ induced by the canonical metric on $\R^{n+1}$ satisfies
$$ ds^2=2\|G_z\|^2|dz|^2.$$
We note that, $g$ is holomorphic since $S$ is minimal.

\begin{proof}[Proof of Theorem \ref{1.1}]
We consider the holomorphic curve $f$ from $S$ into $\P^n(\C)$ with the local reduced representations $F_z=(f_0,\ldots,f_n)$, where $f_j=Q_j^d(g_0,\ldots,g_n)$ and $G_z=(g_0,\ldots,g_n)$ is a local reduced representation of $g$ on a holomorphic local chart $(U,z)$ of $S$.   Denote by $H_j\ (0\le j\le n+1)$ the hyperplane of $\P^n(\C)$ defined by the following linear forms (denoted again by $H_j$):
$$H_j(x_0,\ldots,x_n)=
\begin{cases}
x_j&\text{ if }0\le j\le n,\\
\sum_{i=0}^na_ix_i&\text{ if }j=n+1.
\end{cases}
 $$
From the assumption of the theorem, we have
$$D^{n}_{g}(Q)>\dfrac{(2\ell+1)n(n+1)}{2d\ell}.$$
Then, there exist a compact subset $K\subset S$, a real number $\theta>0$ such that
$$1-\theta>\dfrac{(2\ell+1)n(n+1)}{2d\ell}=\frac{(2\ell+1)\sigma_n}{d\ell},$$
a divisor $\nu$ and a bounded continuous function $k$ (we may assume $0\le k<1$) with mild singularities on $S\setminus K$ such that: $\nu\ge c'$ on $\supp\nu$ for a positive constant $c'$ and
$$ [\min\{\nu_{Q(f)},n\}]+[v]=d\ell\theta\Omega_g +dd^c[\log k^2],$$
i.e.,
$$ [\min\{\nu_{H_{n+1}(F)},n\}]+[v]=d\ell\theta\Omega_g +dd^c[\log k^2]$$
in the sense of currents. We set $h_z=k\|F_z\|^{\theta}$. One has
$$ \nu_{h_{z}}\ge c\text{ on }\supp\nu_{h_{z}}, \text{ where }c=\min\{1,c'\}. $$
Choose a rational number $\epsilon\ (>0)$ such that
$$\dfrac{d\ell(1-\theta)-(2\ell+1)\sigma_n}{d\ell\sigma_{n+1}+\tau_n+d\ell\theta}>\epsilon>\dfrac{d\ell(1-\theta)-(2\ell+1)\sigma_n}{c+d\ell\sigma_{n+1}+\tau_n+d\ell\theta}$$
and define
\begin{align*}
u&:=d\ell(1-\theta)-2\ell\sigma_n-\epsilon d\ell\left (\sigma_{n+1}+\theta\right)>\sigma_n+\epsilon\tau_n,\\ 
\rho&:=\dfrac{1}{u}(\sigma_n+\epsilon\tau_n),\\
\rho^*&:=\dfrac{1}{(1-\rho)u}=\dfrac{1}{d\ell(1-\theta)-(2\ell+1)\sigma_n-\epsilon(d\ell\sigma_{n+1}+\tau_n+d\ell\theta)}. 
\end{align*}
It is clear that $0<\rho<1$ and $c\epsilon\rho^*>1.$

For each $p\ (1\le p\le n)$, one can choose $i_1, \ldots, i_p$ with $0 \leq i_1<\cdots<i_p \leq n$ such that
$$(\psi(F_z)_{jp})_z:=\sum_{l \neq i_1, \ldots, i_p} a_{jl} W_z\left(f_l, f_{i_1}, \cdots, f_{i_p}\right) \not \equiv 0.$$
We consider the set
$$ A_1=\{a\in A; (\psi (F_z)_{jp})_z(a)\ne 0,h_{z}(a)\ne 0\ \forall j=0,\ldots,n+1; p=0,\ldots,n-1\},$$
where $A=S\setminus K$, $z$ is a local holomorphic coordinate around $a$, and define a new pseudo-metric on $A_1$ as follows
$$ d\tau^2=\left (\dfrac{\prod_{j=0}^n\|F_z(H_j)\|^{1-\frac{n}{d}}.\|F_z(H_{n+1})\|}{|((F_z)_{n})_z)|^{1+\epsilon}h_{z}^{1+\epsilon}\prod_{j,p}|(\psi(F_z)_{jp})_z|^{\frac{\epsilon}{n+2}}}\right)^{2\rho^*}|dz|^2.$$
From now on, we will write $\prod_{j,p}$ for $\prod_{j=0}^{n+1}\prod_{p=0}^{n-1}$ for simplicity.
Note that the definitions of $A_1$ and $d\tau^2$ do not depend on the choice of the local holomoprhic coordinates.
Indeed, we have
\begin{align*}
&G_z=(\xi'_z)\cdot G_\xi,\ F_z=(\xi'_z)^{d\ell}\cdot F_\xi,\\
&F_z(H_j)=(\xi'_z)^{d\ell}\cdot F_\xi(H_j),\ v_j(F_z)=(\xi'_z)^{d\ell}\cdot v_j(F_\xi),\\
&\ h_{z}=|\xi'_z|^{d\ell\theta}\cdot h_{\xi},((F_z)_{p})_z=(\xi'_z)^{d\ell(p+1)+\sigma_p}\cdot((F_\xi)_{p})_\xi,\\
&(\psi (F_z)_{ip})_z=(\xi'_z)^{d\ell(p+1)+\sigma_p}\cdot(\psi(F_z)_{ip})_\xi\ (0\le p\le n-1),
\end{align*}
and hence $(\psi (F_z)_{ip})_z(a)=0$ if and only if $(\psi (F_\xi)_{ip})_\xi(a)=0$ for two local holomorphic coordinates $z$ and $\xi$ around $a$. On the other hand, we have
\begin{align*}
&\left (\dfrac{\prod_{j=0}^n\|F_z(H_j)\|^{1-\frac{n}{d}}.\|F_z(H_{n+1})\|}{|((F_z)_{n})_z)|^{1+\epsilon}h_{z}^{1+\epsilon}\prod_{j,p}|(\psi(F_z)_{jp})_z|^{\frac{\epsilon}{n+2}}}\right)^{2\rho^*}|dz|^2\\
=&\left (\dfrac{\prod_{j=0}^n\|F_\xi(H_j)\|^{1-\frac{n}{d}}.\|F_\xi(H_{n+1})\|.(\xi'_z)^{(n+2)d\ell-2\sigma_n\ell}}{|((F_\xi)_{n})_\xi)|^{1+\epsilon}h_{\xi}^{1+\epsilon}\prod_{j,p}|(\psi(F_\xi)_{jp})_\xi|^{\frac{\epsilon}{n+2}}.(\xi'_z)^{(1+\epsilon)(d\ell(n+1)+\sigma_n+d\ell\theta)+\epsilon\sum_{p=0}^{n-1}(d\ell(p+1)+\sigma_p)}}\right)^{2\rho^*}|dz|^2\\
=&\left (\dfrac{\prod_{j=0}^n\|F_\xi(H_j)\|^{1-\frac{n}{d}}.\|F_\xi(H_{n+1})\|}{|((F_\xi)_{n})_\xi)|^{1+\epsilon}h_{\xi}^{1+\epsilon}\prod_{j,p}|(\psi(F_\xi)_{jp})_\xi|^{\frac{\epsilon}{n+2}}}\right)^{2\rho^*}|d\xi|^2
\end{align*}
(since $|\xi'_z|^{2\rho^*((n+2)d\ell-2\sigma_n\ell-(1+\epsilon)(d\ell(n+1)+\sigma_n+d\ell\theta)-\epsilon\sum_{p=0}^{n-1}(d\ell(p+1)+\sigma_p))}|dz^2|=|\xi'_z|^2|dz^2|=|d\xi^2|$).
Then the pseudo-metric $d\tau^2$ is global well-defined.


Since $H_j(F_z),((F_z)_{n})_z,(\psi(F_z)_{jp})_z$ are holomorphic and $k$ are harmonic on $A_1$, $d\tau$ is flat on $A_1$. So, $d\tau$ can be smoothly extended over $K$ to a metric, still call it $d\tau$, on $A_1' =A_1\cup K,$ which is flat outside the compact set $K$.

\begin{claim}\label{cl4.2}
$d\tau$ is complete on $A_1'.$
\end{claim}
Indeed, suppose in contrast that $A_1'$ is not complete with $d\tau$. Then there is a divergent curve $\gamma: [0,1)\rightarrow A_1'$ with finite length. By using only the last segment of $\gamma$ if necessary, we may assume that the distance $d(\gamma,K)$ between $\gamma$ and $K$ exceeds a positive number $d^*$. Then, as $t\rightarrow 1$ there are only two cases: either $\gamma(t)$ tends to a point $a$ with
$$ \prod_{i=0}^{n+1}(h_z\prod_{p=0}^{n-1}|\psi(F_z)_{ip}|)(a)=0$$
for a local holomorphic coordinate $z$ around $a$, or else $\gamma (t)$ tends to the boundary of $S$.

For the first case, by Theorem \ref{thm3.5}, we have
\begin{align*}
\nu_{d\tau}(a)&=\biggl(\big(1-\dfrac{n}{d}\big)\sum_{j=0}^n\nu_{F_z(H_j)}(a)+\nu_{F_z(H_{n+1})}(a)\\
&\  \ -(1+\epsilon)\nu_{((F_z)_{n})_z}(a)-(1+\epsilon)\nu_{h_z}(a)-\dfrac{\epsilon}{n+2}\sum_{j=0}^{n+1}\sum_{p=0}^{n-1}\nu_{\psi(F_z)_{jp}}(a)\biggl)\rho^*\\
&= -\biggl (\nu_{((F_z)_{n})_z}(a)-\big(1-\dfrac{n}{d}\big)\sum_{j=0}^{n}\nu_{F_z(H_j)}(a)-\nu_{F_z(H_{n+1})}(a)+\min\{\nu_{F_z(H_{n+1})},n\}\\ 
&\  \ + (v_{h_z}-\min\{\nu_{F_z(H_{n+1})},n\} )+\bigl(\epsilon\nu_{((F_z)_{n})_z(a)}+\dfrac{\epsilon}{n+2}\sum_{j=0}^{n+1}\big(\nu_{h_{z}}+\sum_{p=0}^{n-1}\nu_{\psi(F_z)_{jp}}(a)\big)\bigl)\biggl)\rho^*\\
&\le -\epsilon\rho^*\nu_{((F_z)_{n})_z}(a)-\dfrac{\epsilon\rho^*}{n+2}\sum_{j=0}^{n+1}\big(\nu_{h_{z}}+\sum_{p=0}^{n-1}\nu_{\psi(F_z)_{jp}}(a)\big)\le -c\epsilon\rho^*
\end{align*}
(since $\nu_{h_{z}}>c$ on $\supp\nu_{h_z}$). Then, there is a positive constant $C$ such that
$$ |d\tau|\ge\dfrac{C}{|z-z(a)|^{c\epsilon\rho^*}}|dz|$$
in a neighborhood of $a$. Then we get
$$ L_{d\tau}(\gamma)=\int_0^1\|\gamma'(t)\|_{\tau}dt=\int_{\gamma}d\tau \ge\int_\gamma \dfrac{C}{|z-z(a)|^{c\epsilon\rho^*}}|dz|=+\infty$$
($\gamma (t)$ tends to $a$ as $t\rightarrow 1$). This is a contradiction. Therefore, the second case must occur, that is $\gamma (t)$ tends to the boundary of $S$ as $t\rightarrow 1$.  

Choose $t_0$ close $1$ enough so that $L_{d\tau}(\rho|_{(t_0,1)})<d^*/3$. Take a disk $\Delta$ around $\gamma(t_0)$ in the metric induced by $d\tau$. Since $d\tau$ is flat, by Lemma \ref{lem4.1}, there is an isometric $\Phi$ from an ordinary disk $\Delta(r)=\{|\omega|<r\}\subset\C$ to $\Delta$  with $\Phi(0)=\gamma(t_0)$. Extend $\Phi$ as a local isometric into $A_1$ to the largest disk possible $\Delta(R)=\{|\omega|<R\}$, and denote again by $\Phi$ this extension (for simplicity, we may consider $\Phi$ as the exponential map).  

Since $\int_{\gamma|_{[t_0,1)}}d\tau<d^*/3$, we have $R\le d^*/3$. Hence, the image of $\Phi$ is bounded away from $K$ by distance at least $2d^*/3$. Since $\Phi$ cannot be extended to a larger disk, the image of $\Phi$ must go to the boundary $A_1$. But, this image cannot go to points $a$ of $A'_1$ with $ \prod_{i=0}^{n+1}(h_z\prod_{p=0}^{n-1}|\psi(F_z)_{ip}|)(a)=0$, since we have already shown that $\gamma(t_0)$ is infinitely far away in the metric with respect to these points. Then the image of $\Phi$ must go to the boundary $S$. Hence, by again Lemma \ref{lem4.1}, there exists a point $w_0$ with $|w_0|= R$ so that $\Gamma=\Phi(\overline{0,w_0})$ is a divergent curve on $S$.

We now show that $\Gamma$ has finite length in the original metric $ds^2$ on $S$, which is contrary to the completeness of $S$. Let $\tilde g:=g\circ\Phi :\Delta(R)\rightarrow \P^n(\C)$ be a holomorphic curve. Indeed, for a local holomorphic coordinate $z$ on the image of $\Phi$, the map $\tilde g$ has a local reduced representation
$$\tilde G=(\tilde g_0,\ldots,\tilde g_n),$$
where $\tilde g_i=g_i\circ\Phi\ (0\le i\le n).$ We define $\tilde f=f\circ\Phi$ and $\tilde F=F_z\circ\Phi$.
Hence, we have:
\begin{align*}
\Phi^*ds^2&=2\|G\circ\Phi\|^2|\Phi^*dz|^2=2\|\tilde G\|^2\left|\dfrac{d(z\circ\Phi)}{dw}\right|^2|dw|^2,\\
(\tilde F_{n})_w&=((F_z\circ\Phi)_{n})_w=((F_z)_{n})_z\circ\Phi\cdot\left(\dfrac{d(z\circ\Phi)}{dw}\right)^{\sigma_n},\\
(\psi(\tilde F)_{jp})_w&=(\psi(F_z\circ\Phi)_{jp})_w=(\psi(F_z)_{jp})_z\cdot\left(\dfrac{d(z\circ\Phi)}{dw}\right)^{\sigma_p}, (0\le j\le n+1,0\le p\le n-1).
\end{align*}
On the other hand, since $\Phi$ is locally isometric,
\begin{align*}
|dw|&=|\Phi^*d\tau|\\
&= \left (\dfrac{\prod_{j=0}^n\|F_z(H_j)\circ\Phi\|^{1-\frac{n}{d}}.\|F_z(H_{n+1})\circ\Phi\|}{|((F_z)_{n})_z\circ\Phi)|^{1+\epsilon}|h_{z}\circ\Phi|^{1+\epsilon}\prod_{j,p}|(\psi(F_z)_{jp})_z\circ\Phi|^{\frac{\epsilon}{n+2}}}\right)^{\rho^*}\biggl|\dfrac{d(z\circ\Phi)}{d\omega}\biggl||d\omega|\\
&= \left (\dfrac{\prod_{j=0}^n\|\tilde F(H_j)\|^{1-\frac{n}{d}}.\|\tilde F(H_{n+1})\|}{|(\tilde F_n)_{\omega}|^{1+\epsilon}|h_{z}\circ\Phi|^{1+\epsilon}\prod_{j,p}|(\psi(\tilde F)_{jp})_\omega|^{\frac{\epsilon}{n+2}}}\right)^{\rho^*}\biggl|\dfrac{d(z\circ\Phi)}{d\omega}\biggl|^{1+u\rho\rho^*}|d\omega|
\end{align*}
(because $1+((1+\epsilon)\sigma_n+\epsilon\sum_{p=0}^{n-1}\sigma_p)\rho^*=1+u\rho\rho^*$).
This implies that
\begin{align*}
\left|\dfrac{d(z\circ\Phi)}{dw}\right|&=\left (\dfrac{|(\tilde F_n)_{\omega}|^{1+\epsilon}|h_{z}\circ\Phi|^{1+\epsilon}\prod_{j,p}|(\psi(\tilde F)_{jp})_\omega|^{\frac{\epsilon}{n+2}}}{\prod_{j=0}^n\|\tilde F(H_j)\|^{1-\frac{n}{d}}.\|\tilde F(H_{n+1})\|}\right)^{\frac{\rho^*}{1+u\rho\rho^*}}\\
&\le \left (\dfrac{|(\tilde F_n)_{\omega}|^{1+\epsilon}|h_{z}\circ\Phi|^{1+\epsilon}\prod_{j,p}|(\tilde F_p)_\omega(H_j)|^{\frac{\epsilon}{n+2}}}{\prod_{j=0}^n\|\tilde F(H_j)\|^{1-\frac{n}{d}}.\|\tilde F(H_{n+1})\|}\right)^{\frac{\rho^*}{1+u\rho\rho^*}}\\
&=\left (\dfrac{|(\tilde F_n)_{\omega}|^{1+\epsilon}|h_{z}\circ\Phi|^{1+\epsilon}\prod_{j,p}|(\tilde F_p)_\omega(H_j)|^{\frac{\epsilon}{n+2}}}{\prod_{j=0}^n\|\tilde F(H_j)|^{1-\frac{n}{d}}.\|\tilde F(H_{n+1})\|}\right)^{1/u}.
\end{align*}
Hence, we have
\begin{align*}
\Phi^*ds&\le\sqrt{2}\|\tilde G\|\left (\dfrac{|(\tilde F_n)_{\omega}|^{1+\epsilon}|h_{z}\circ\Phi|^{1+\epsilon}\prod_{j,p}|(\tilde F_p)_\omega(H_j)|^{\frac{\epsilon}{n+2}}}{\prod_{j=0}^n\|\tilde F(H_j)\|^{1-\frac{n}{d}}.\|\tilde F(H_{n+1})\|}\right)^{1/u}|dw|\\
&=B\left (\dfrac{\|\tilde F_0\|^{\frac{u}{d\ell}}|(\tilde F_n)_{\omega}|^{1+\epsilon}|h_{z}\circ\Phi|^{1+\epsilon}\prod_{j,p}|(\tilde F_p)_\omega(H_j)|^{\frac{\epsilon}{n+2}}}{\prod_{j=0}^n\|\tilde F(H_j)\Phi\|^{1-\frac{n}{d}}.\|\tilde F(H_{n+1})\|}\right)^{1/u}|dw|,
\end{align*}
for a positive constant $B$.
We note that $\dfrac{u}{d\ell}=1-\theta-\dfrac{n(n+1)}{d}-\epsilon(\sigma_{n+1}+\theta)$. Then the inequality (\ref{lem3.7}) yields that the conditions of Lemma \ref{ML} are satisfied. Then, by applying Lemma \ref{ML} we have
$$ \Phi^*ds\le C\left (\dfrac{2R}{R^2-|w|^2}\right)^\rho|dw|,$$
for a positive constant $C$. Also, we have $0<\rho<1$. Then
$$L_{ds^2}(\Gamma)\le\int_{\Gamma}ds=\int_{\overline{0,w_0}}\Phi^*ds\le C\cdot\int_{0}^R\left(\dfrac{2R}{R^2-|w|^2}\right)^{\rho}|dw|<+\infty. $$
This is contrary to the assumption of the completeness of $S$ with respect to $ds^2$. Thus, Claim \ref{cl4.2} is proved.

We note that the metric $d\tau^2$ on $A_1'$ is flat outside $K$, and $A_1'$ is complete with this metric by Claim \ref{cl4.2}. Then $A_1'$ has finite total curvature. By the theorem \cite[Theorem 13, p.61]{Hu61} of  Huber (also see \cite[Theorem 4.8]{Fu91}), it follows that $A_1'$ is finite type (i.e., biholomorphic with a compact Riemann surface from which finitely many points have been removed). Hence, there is a compact subregion of $A_1'$ whose complement is the union of a finite number of doubly-connected regions. Therefore, the functions $\prod_{j=0}^{n+1}\prod_{p=0}^{n-1}|\psi(F_z)_{jp}|$ must have only a finite number of zeros. Hence, $S$ is conformally equivalent to a compact Riemann surface $\overline{S}$ with a finite number of points removed. By the assumption, in a neighborhood of each of those points the Gauss map $g$ satisfies
$$ \delta_{g}^{H}(Q)> \dfrac{n(n+1)}{d}.$$
Then, Lemma \ref{cor2.9} yields that the Gauss map $g$ is not essential singular at those points. In other words, $g$ can be extended to a holomorphic map from $\overline{S}$ to $\P^n(\C)$. If the homology class represented by the image of $g:S\rightarrow V\subset\P^n(\C)$ is $v$ times the fundamental homology class of $\P^n(\C)$, then the total curvature of $S$ is defined by
$$C(S)=\iint\mathcal K(s)dS=-2\pi v,$$
where $\mathcal K(s)$ is the Gaussian curvature of $S$. Hence Theorem \ref{1.1} is proved.
\end{proof}

\begin{proof}[Proof of Theorem \ref{1.3}]
 Suppose contrarily that
$$\delta^{S,n}_{g,S}(Q)>\frac{n(n+1)}{d}+\frac{n(n+1)}{2\ell d}.$$
Let $G=(g_0,\ldots,g_n)$ be a global reduced representation of $g$.  Then, there are a constant $\theta$ and a continuous subharmonic functions $u$ on $S$ satisfying conditions $(D_1),(D_2)$ with respect to the number $\theta$ such that
$$1-\theta> \frac{n(n+1)}{d}+\frac{n(n+1)}{2\ell d}.$$
Let $f$ be the holomorphic curve from $S$ into $\P^n(\C)$ with the reduced representation $F=(Q_0^d(G),\ldots,Q_n^d(G))$. We may assume that $\|G\|^{d\ell}\le \|F\|.$ We use the same hyperplanes $H_0,\ldots,H_{n+1}$ as similar as in the proof of Theorem \ref{1.1}.

%
%
Set:
$$\gamma:= (1-\theta)-\frac{n(n+1)}{d}>\frac{\sigma_n}{d\ell}.$$
Then, we may choose a rational number $\epsilon\ (>0)$ such that
$$\gamma-\epsilon\sigma_{n+1}=\frac{\sigma_n+\epsilon\tau_n}{d\ell}.$$
On a local holomorphic chart $(U,z)$ of $S$, we set
$$ \lambda_z=\left( \dfrac{|F_{0}|^{\gamma-\epsilon\sigma_{n+1}}e^{u}|(F_n)_z|\prod_{p=0}^{n}|(F_{p})_z|^{\epsilon}}{|F(H_{n+1})|\prod_{j=0}^n(|F(H_j)|^{1-\frac{n}{d}}\prod_{j,p}\log (\delta/\varphi_{z,p}(H_j)))}\right )^{\frac{1}{\sigma_n+\epsilon\tau_n}},$$
where $\varphi_{z,p}(H_j)=\dfrac{|(F_{p})_z(H_j)|}{|(F_{p})_z|}, j=0,\ldots,n+1,$ and $\delta$ is the number satisfying the conclusion of Theorem \ref{thm3.3} with $F,\varphi_{z,p}(H_j)$ in place of $F,\varphi_{p}(H_j)$.

We define a pseudo-metric $d\tau^2:=\lambda_z^2|dz|^2$ on each chart $(U,z)$ of $S$. We will show that $d\tau^2$ is well-defined, i.e., not depend on the choice of the local coordinate. Indeed, if we have another local coordinate $\xi$, we have
$(F_p)_z=(\xi'_z)^{\sigma_p}\cdot (F_p)_\xi\ (0\leq p\le n+1),$ and hence
\begin{align*}
\lambda_{z}&=\left( \dfrac{|F_{0}|^{\gamma-\epsilon\sigma_{n+1}}e^u|(F_n)_z|\prod_{p=0}^{n}|(F_p)_z|^{\epsilon}}{|F(H_{n+1})|\prod_{j=0}^n(|F(H_j)|^{1-\frac{n}{d}}\prod_{j,p}\log (\delta/\varphi_{z,p}(H_j)))}\right )^{\frac{1}{\sigma_n+\epsilon\tau_n}}\\
    &=\left( \dfrac{|F_0|^{\gamma-\epsilon\sigma_{n+1}} e^u|(F_n)_\xi|\prod_{p=0}^{n}|(F_p)_\xi|^{\epsilon}|\xi_z'|^{\sigma_n+\epsilon\tau_n}}{|F(H_{n+1})|\prod_{j=0}^n(|F(H_j)|^{1-\frac{n}{d}}\prod_{j,p}\log (\delta/\varphi_{\xi,p}(H_j)))}\right )^{\frac{1}{\sigma_n+\epsilon\tau_n}}\\
&=\left( \dfrac{|F_0|^{\gamma-\epsilon\sigma_{n+1}}e^u|(F_n)_\xi|\prod_{p=0}^{n}|(F_p)_\xi|^{\epsilon}}{|F(H_{n+1})|\prod_{j=0}^n(|F(H_j)|^{1-\frac{n}{d}}\prod_{j,p}\log (\delta/\varphi_{\xi,p}(H_j)))}\right )^{\frac{1}{\sigma_n+\epsilon\tau_n}}|\xi'_z|.
\end{align*}
Hence $\lambda_z^2|dz^2|=\lambda_\xi^2|\xi'_z|^2|dz^2|=\lambda_\xi^2|d\xi^2|$. Then, $d\tau^2$ is well-defined.

It is obvious that $d\tau^2$ is continuous on $S\setminus\bigcup_{j=1}^q f^{-1}(H_j)$. Take a point $a$ such that $\prod_{j=1}^q H_j(F)=0$ (for a loal holomorphic coordinate $z$ around $a$). We have
$$ \lim_{t\rightarrow a}(e^u(t)\cdot |z(t)-z(a)|^{-\min\{n,\nu_{H_{n+1}(F)}(a)\}})<\infty.$$
With the help of Theorem \ref{thm3.5}, this implies that
\begin{align*}
d\tau_z(a)&\ge\dfrac{1}{\sigma_n+\epsilon\tau_n}\biggl (\nu_{(F_n)_z}-\sum_{j=0}^n(1-\frac{n}{d})\nu_{H_j(F)}+\min\{n,\nu_{H_{n+1}(F)}(a)\}-\nu_{H_{n+1}(F)}(a))\biggl)\\ 
&\ge\dfrac{1}{\sigma_n+\epsilon\tau_n}\biggl (\nu_{(F_n)_z}-\sum_{j=0}^{n+1}(\nu_{H_j(F)}(a)-\min\{n,\nu_{H_{n+1}(F)}(a)\})\biggl) \ge 0.
\end{align*}
Therefore $d\tau$ is continuous at $a$. This yields that $d\tau$ is a continuous pseudo-metric on $S$.

We now prove that $d\tau^2$ has strictly negative curvature. Similarly as above, we have
\begin{align*}
dd^c&\left[\log\dfrac{|(F_n)_z|}{|H_{n+1}(F)\prod_{j=0}^{n}|H_j(F)|^{1-\frac{n}{d}}}\right]+dd^c[\log e^u]\\
&\ge dd^c[\log e^u]+\sum_{i=0}^n\frac{n}{d}[\nu_{H_j(F)}]-\sum_{i=0}^{n+1}[\min\{n,\nu_{H_j(F)}\}]\ge 0.
\end{align*}
Then by Theorems \ref{thm3.3} and \ref{thm3.4}, similarly as (\ref{new3}) we have
\begin{align*}
dd^c\log\lambda_z&\ge\dfrac{\gamma-\epsilon\sigma_{n+1}}{\sigma_n+\epsilon\tau_M}d\ell\Omega_g+\dfrac{\epsilon}{2(\sigma_n+\epsilon\tau_n)}dd^c\log\left(|(F_0)_z|\cdots|(F_n)_z|\right)\\
& +\dfrac{1}{2(\sigma_n+\epsilon\tau_n)}dd^c\log\dfrac{\prod_{p=0}^{n-1}|(F_p)_z|^{2\epsilon}}{\sum_{j=0}^{n+1}\prod_{p=0}^{n-1}\log^{4}(\delta/\varphi_{z,p}(H_j))}\\
&\ge C\left (\dfrac{|(F_0)_z|^{1-\epsilon\sigma_{n+1}}|(F_n)_z|\prod_{p=0}^n|(F_p)_z|^\epsilon}{\prod_{j=0}^{n+1}(|H_j(F)|\prod_{p=0}^{n-1}\log(\delta/\varphi_{z,p}(H_j)))}\right)^{\frac{2}{\sigma_n+\epsilon\tau_n}}dd^c|z|^2
\end{align*}
for a positive constant $C$. On the other hand, we have $e^u\le \|G\|^{d\ell\theta}\le\|F\|^{\theta}$, and hence
\begin{align*}
|(F_{0})_z|^{1-\epsilon\sigma_{n+1}}&=|F_{0}|^{\gamma-\epsilon\sigma_{n+1}+\theta+\frac{n(n+1)}{d}}\\
&\ge |F_{0}|^{\gamma-\epsilon\sigma_{n+1}}e^u\cdot\prod_{j=0}^n|H_j(F)|^{\frac{n}{d}}.
\end{align*}
This implies that $\Delta \log\eta^2\ge C\eta^2$. Therefore, $d\tau^2$ has strictly negative curvature.

As we known that the universal covering surface of $S$ is biholomorphic to $\C$ or a disk in $\C$. If the universal covering of $S$ is holomorphic to $\C$ then from \cite[Lemma 3]{QT} we have
$$\delta^{S,n}_{g,S}(Q)\le\delta_g^n(Q)\le \dfrac{n(n+1)}{d},$$
which is contrary to the supposition. Now, we consider the case where the universal covering surface $\tilde S$ of $S$ is biholomorphic to the unit disc $\Delta$. Without loss of generality, we may assume that $\tilde S=\Delta$ and let $\Phi :\Delta\rightarrow S$ be the covering map. From Lemma \ref{lem3.6} we have
$$ \Phi^*d\tau^2\le d\sigma_{\Delta}^2,$$
where $d\sigma_{\Delta}=\dfrac{2|dz|^2}{1-|z|^2}$ with the complex coordinate $z$ on $\Delta$.

Since $S$ has total finite curvature, $S$ is conformally equivalent to a compact surface $\overline S$ punctured at a finite number of points $P_1,\ldots,P_m$. Take disjoint neighborhoods $U_i$ of $P_i\ (1\le i\le r)$ in $\overline S$ and biholomorphic maps $\phi_i:U_i\rightarrow\Delta$ with $\phi_i(P_i)=0$. Note that, the Poincare-metric on $\Delta^*=\Delta\setminus\{0\}$ is given by $d\sigma_{\Delta^*}^2=\dfrac{4|dz|^2}{|z|^2\log^2|z|^2}$. Hence, by the decreasing distance property, one has
$$\Phi^*d\tau^2\le d\sigma_\Delta^2\le C\cdot(\Phi\circ\phi_i^{-1})^*d\sigma_{\Delta^*}^2\ (1\le i\le r)$$
for a positive constant $C$. This implies that
$$\int_{U_i}\Omega_{d\tau^2}\le \int_{\Phi^{-1}(U_i)}\Phi^*\Omega_{\sigma_\Delta^2}\le vC\int_{\Delta^*}\Omega_{d\sigma_{\Delta^*}^2}<\infty.$$
where $v$ is the number of the sheets of the covering $\Phi$. This yields that
$$\int_S\Omega_{d\tau^2}\le \int_{S\setminus\bigcup_{i=1}^rU_i}\Omega_{d\tau^2}+vCr\int_{\Delta^*}\Omega_{d\sigma_{\Delta^*}^2}<\infty.$$

Now, denote by $ds^2$ the original metric on $S$. Similar as (\ref{new3}), we have
$$ dd^c\log\lambda_z \ge\dfrac{\gamma-\epsilon\sigma_{n+1}}{\sigma_n+\epsilon\tau_n}d\ell\Omega_g.$$ 
Then there is a subharmonic function $v_z$ such that
\begin{align*}
\lambda_z^2|dz|^2&=e^{v_z}\|G_z\|^{2d\ell\frac{\gamma-\epsilon\sigma_{n+1}}{\sigma_n+\epsilon\tau_n}}|dz|^2=e^{v_z}\|G_z\|^{2\frac{d\ell(\gamma-\epsilon\sigma_{n+1})-\sigma_n-\epsilon\tau_n}{\sigma_n+\epsilon\tau_n}}\|G_z\|^2 |dz|^2=e^w ds^2
\end{align*}
for a subharmonic function $w$ on $S$. Since $S$ is complete with respect to $ds^2$, by a result of S. T. Yau \cite{Y76} and L. Karp \cite{K82} we have
$$ \int_S\Omega_{d\tau^2}=\int_Se^w\Omega_{ds^2}=+\infty.$$
This is a contradiction.

Then we have the derised inequality
$$\delta^{S,n}_{g,S}(Q)\le\frac{n(n+1)}{d}+\frac{n(n+1)}{2\ell d}.$$
The theorem is proved.
\end{proof}

\end{document}